\documentclass[11pt]{amsart}
\theoremstyle{plain}
\newtheorem{Thm}{Theorem}

\newtheorem{Lem}[Thm]{Lemma}

\begin{document}

\title[Q-curvature flow]
{Q-curvature flow with indefinite nonlinearity}

\author{Li MA}
\address{Li Ma, Department of mathematical sciences \\
Tsinghua University \\
Beijing 100084 \\
China} \email{lma@math.tsinghua.edu.cn}\dedicatory{}

\date{July 12th, 2008}

\begin{abstract}

In this note, we study Q-curvature flow on $S^4$ with indefinite
nonlinearity. Our result is that the prescribed Q-curvature problem
on $S^4$ has a solution provided the prescribed Q-curvature $f$ has
its positive part, which possesses non-degenerate critical points
such that $\Delta_{S^4} f\not=0$ at the saddle points and an extra
condition such as a nontrivial degree counting condition.

{ \textbf{Mathematics Subject Classification 2000}: 53Cxx,35Jxx}

{ \textbf{Keywords}: Q-curvature flow, indefinite nonlinearity,
blow-up, conformal class}
\end{abstract}

\thanks{$^*$ The research is partially supported by the National Natural Science
Foundation of China 10631020 and SRFDP 20060003002 }
 \maketitle

\section{Introduction}

Following the works of A.Chang-P.Yang \cite{chang}, M.Brendle
\cite{brandle}, Malchiodi and M.Struwe \cite{struwe}, we study a
heat flow method to the prescribed Q-curvature problem on $S^4$.
Given the Riemannian metric $g$ in the conformal class of standard
metric $c$ on $S^4$ with Q-curvature $Q_g$. Then it is the
well-known that
$$
Q_g=-\frac{1}{12}(\Delta_gR_g-R_g^2+3|Rc(g)|^2):=Q,
$$
where $R_g$, $Rc(g)$, $\Delta_g$ are the scalar curvature, Ricci
curvature tensor, the Laplacian operator of the metric $g$
respectively.

Recall the Chern-Gaussian-Bonnet formula on $S^4$ is
$$
\int_{S^4} Q_gdv_g=8\pi^2,
$$
Hence,we know that $Q_g$ has to be positive somewhere. This gives
a necessary condition for the prescribed Q-curvature problem on
$S^4$. Assuming the prescribed curvature function $f$ being
positive on $S^4$, the heat flow for the Q-curvature problem is a
family of metrics of the form $g=e^{2u(x,t)}c$ satisfying
\begin{equation}\label{flow}
 u_t=\alpha f-Q, \quad x\in S^4,
\quad t>0,
\end{equation}
where $u:S^4\times (0,T)\to R$, and $\alpha=\alpha(t)$ is defined
by
\begin{equation}\label{alpha}
\alpha \int_{S^4} fdv_g=8\pi^2. \end{equation}
 Here $dv_g$ is the
area element with respect to the metric $g$. It is easy to see
that
$$
\alpha_t\int_{S^4}fdv_g=2\alpha \int_{S^4}(Q-\alpha f)fdv_g.
$$
 Malchiodi and M.Struwe \cite{struwe} can
show that the flow exists globally, furthermore, the flow converges
at infinity provided $f$ possesses non-degenerate critical points
such that $\Delta_{S^4} f\not=0$ at the saddle points with the
condition $$\sum_{\{p:\nabla f(p)=0;
\Delta_{S^4}f(p)<0\}}(-1)^{ind(f,p)}\not=0.$$
 Here $\Delta_{S^4}:=\Delta$ is the
Analyst's Laplacian on the standard 4-sphere $(S^4,c)$. Recall
that $\int_{S^4}dv_c=\frac{8}{3}\pi^2$. The purpose of this paper
is to relax their assumption by allowing the function $f$ to have
sign-changing or to have zeros.

Since we have
$$
Q=\frac{1}{2}e^{-4u}(\Delta^2
u-div((\frac{2}{3}R(c)c-2Rc(c))du)+6),
$$
the equation (\ref{flow}) define a nonlinear parabolic equation
for $u$, and the flow exists at least locally for any initial data
$u|_{t=0}=u_0$. Clearly, we have $$
\partial_t\int_{S^4}dv_g=2\int{S^4}u_tdv_g=0.
$$
We shall assume that the initial data $u_0$ satisfies the
condition
\begin{equation}\label{positive}
\int f e^{4u}dv_c>0.
\end{equation}
We shall show that this property is preserved along the flow. It
is easy to compute that
\begin{equation}\label{curvature}
Q_t=-4u_tQ-\frac{1}{2}Pu_t=4Q(Q-\alpha f)+P(\alpha f-Q),
\end{equation}
where $P=P_g=e^{-4u}P_c$ and $P_c$ is the Paneitz operator in the
metric $c$ on $S^4$ \cite{chang}. Using (\ref{curvature}), we can
compute the growth rate of the Calabi-type energy
$\int_{S^4}|Q-\alpha f|^2dv_g$.

Our main result is following
\begin{Thm}\label{main}
Let $f$ be a positive somewhere, smooth function on $S^4$ with only
non-degenerate critical points on the its positive part $f_+$ with
its Morse index $ind(f_+,p)$. Suppose that at each critical point
$p$ of $f_+$, we have $\Delta f\not=0$. Let $m_i$ be the number of
critical points with $f(p)>0$, $\Delta _{S^4}f(p)<0$ and
$ind(f,p)=4-i$. Suppose that there is no solutions with coefficients
$k_i\geq 0$ to the system of equations
$$
m_0=1+k_0, m_i=k_{i-1}+k_i, 1\leq i\leq 4, k_4=0.
$$

 Then $f$ is the Q curvature
of the conformal metric $g=e^{2u}c$ on $S^4$.
\end{Thm}

Note that this result is an extension of the famous result of
Malchiodi-Struwe \cite{struwe} where only positive $f$ has been
considered. A similar result for Curvature flow to Nirenberg problem
on $S^2$ has been obtained in \cite{ma2}. See also J.Wei and X.Xu's
work \cite{Wx}.

 For
simplifying notations, we shall use the conventions that
$dc=\frac{dv_c}{\frac{8}{3}\pi^2}$ and $\bar{u}=\bar{u}(t)$
defined by
$$
\int_{S^4}(u-\bar{u})dv_c=0.
$$

\section{Basic properties of the flow}

Recall the following result of Beckner \cite{ono} that
\begin{equation}\label{onof}
\int_{S^4}(|\Delta u|^2+2|\nabla u|^2+12u)dc\geq \log(\int_{S^4}
e^{4u}dc)=0,
\end{equation}
where $|\nabla u|^2$ is the norm of the gradient of the function
$u$ with respect to the standard metric $c$. Here we have used the
fact that $\int_{S^4} e^{4u}dc=1$ along the flow (\ref{flow}).

We show that this condition is preserved along the flow
(\ref{flow}). In fact, letting
$$
E(u)=\int_{S^4}(uPu+4Q_cu)dc=\int_{S^4}(|\Delta u|^2_c+2|\nabla
u|_c^2+12u)dc
$$
be the Liouville energy of $u$ and letting
$$
E_f(u)=E(u)-3\log (\int_{S^4}fe^{4u}dc)
$$
be the energy function for the flow (\ref{flow}), we then compute
that
\begin{equation}\label{monotone}
\partial_tE_f(u)=-\frac{3}{2\pi^2}\int_{S^4}|\alpha f-Q|^2dv_g\leq 0.
\end{equation}
One may see Lemma 2.1 in \cite{struwe} for a proof. Hence
$$
E_f(u(t))\leq E_f(u_0), \quad t>0.
$$
After using the inequality (\ref{onof}) we have
\begin{equation}\label{bound}
\log(1/\int_{S^4}fe^{4u}dc) \leq E_f(u_0),
\end{equation}
which implies that $\int_{S^4} fe^{4u}dv_c>0$ and $$
e^{E_f(u_0)}\int_{S^4}e^{4u}dc\leq \int_{S^4}fe^{4u}dc.
$$
 Note also that
$\int_{S^4}fe^{4u}dc=1/\alpha(t)$. Hence,
$$
\alpha(t)\leq \frac{1}{e^{E_f(u_0)}}.
$$

Using the definition of $\alpha(t)$ we have
$$
\alpha(t)\geq \frac{1}{\max_{S^4}f}.
$$
We then conclude that $\alpha(t)$ is uniformly bounded along the
flow, i.e.,
\begin{equation}\label{struwe2}
\frac{1}{\max_{S^4}f}\leq\alpha(t)\leq \frac{1}{e^{E_f(u_0)}}.
\end{equation}
We shall use this inequality to replace (26) in \cite{struwe} in
the study of the normalized flow, which will be defined soon
following the work of Machiodi and M.Struwe \cite{struwe}. If we
have a global flow, then using (\ref{monotone}) we have
$$
2\int_0^{\infty}dt\int_{S^4} |\alpha f-Q|^2dv_g\leq 4\pi
(E_f(u_0)+\log max_{S^4} f).
$$

Hence we have a suitable sequence $t_l\to\infty$ with associated
metrics $g_l=g(t_l)$ and $\alpha (t_l)\to \alpha>0$, and letting
$Q_l=Q(g_l)$ be the Q-curvature of the metric $g_l$, such that
$$
\int_{S^4}|Q_l-\alpha f|^2\to 0, \; \; (t_l\to\infty).
$$
Therefore, once we have a limiting metric $g_{\infty}$ of the
sequence of the metrics $g_l$, it follows that
$Q(g_{\infty})=\alpha f$. After a re-scaling, we see that $f$ is
the Gaussian curvature of the metric $\beta g_{\infty}$ for some
$\beta>0$, which implies our Theorem \ref{main}.

\section{Normalized flow and the proof of Theorem \ref{main}}

We now introduce a normalized flow. For the given flow
$g(t)=e^{2u(t)}c$ on $S^4$, there exists a family of conformal
diffeomorphisms $\phi=\phi(t):S^4\to S^4$, which depends smoothly
on the time variable $t$, such that for the metrics $h=\phi^*g$,
we have
$$
\int_{S^4} x dv_h=0, \; for \; all \; t\geq 0.
$$
Here $x=(x^1,x^2,x^3, x^4, x^5)\in S^4\subset R^5$ is a position
vector of the standard 4-sphere. Let
$$
v=u\circ \phi+\frac{1}{4}\log(det(d\phi)).
$$
Then we have $h=e^{2v}c$. Using the conformal invariance of the
Liouville energy \cite{chang}, we have
$$
E(v)=E(u),
$$
and furthermore,
$$
Vol(S^4,h)=Vol(S^4,g)=\frac{8}{3}\pi^2, \; for \; all \; t\geq 0.
$$

Assume $u(t)$ satisfies (\ref{flow}) and (\ref{alpha}). Then we
have the uniform energy bounds
$$
0\leq E(v)\leq E(u)=E_f(u)+\log (\int_{S^4}fe^{4u}dc)\leq
E_f(u_0)+\log(\max_{S^4}f).
$$

Using Jensen's inequality we have
$$
2\bar{v}:=\int_{S^4} 2v dc\leq \log(\int_{S^4} e^{4v}dc)=0.
$$

Using this we can obtain the uniform $H^1$ norm bounds of $v$ for
all $t\geq 0$ that
$$
\sup_t|v(t)|_{H^1(S^2)}\leq C.
$$
See the proof of Lemma 3.2 in \cite{struwe}. Using the
Aubin-Moser-Trudinger inequality \cite{Au98} we further have
$$
4\sup_{\{0\leq t<T\}}\int_{S^4}|u(t)|dc\leq \sup_t\int_{S^4}
e^{4|u(t)|}dc\leq C<\infty.
$$

Note that
$$
v_t=u_t\circ \phi+\frac{1}{4} e^{-4v}div_{S^4}(\xi e^{4v})
$$
where $\xi=(d\phi)^{-1}\phi_t$ is the vector field on $S^2$
generating the flow $(\phi(t))$, $t\geq 0$, as in \cite{struwe},
formula (17), with the uniform bound
$$
|\xi|^2_{L^{\infty}(S^4)}\leq C\int_{S^4}|\alpha f-K|^2dv_g.
$$
With the help of this bound, we can show (see Lemma 3.3 in
\cite{struwe}) that for any $T>0$, it holds
$$
\sup_{0\leq t<T}\int_{S^2}e^{4|u(t)|}dc<+\infty.
$$
Following the method of Malchiodi and M.Struwe \cite{struwe} (see
also Lemma 3.4 in  \cite{struwe1}) and using the bound
(\ref{struwe2}) and the growth rate of $\alpha$, we can show that
$$
\int_{S^4}|\alpha f-Q|^2dv_g\to 0
$$
as $t\to\infty$. Once getting this curvature decay estimate, we
can come to consider the concentration behavior of the metrics
$g(t)$. Following \cite{struwe1}, we show that

\begin{Lem} Let $(u_l)$ be a sequence of smooth functions on $S^4$
with associated metrics $g_l=e^{2u_l}c$ with
$vol(S^4,g_l)=\frac{8}{3}\pi^2$, $l=1,2,...$ as constructed above.
Suppose that there is a smooth function $Q_{\infty}$, which is
positive somewhere in $S^4$ such that
$$
|Q(g_l)-Q_{\infty}|_{L^2(S^4,g_l)}\to 0
$$
as $l\to\infty$. Let $h_l=\phi_l^*g_l=e^{2v_l}c$ be defined as
before. Then we have either

 1) for a subsequence $l\to\infty$ we have $u_l\to u_{\infty}$ in
 $H^4(S^4, c)$, where $g_{\infty}=e^{2u_{\infty}}c$ has Q-
 curvature $Q_{\infty}$, or

 2) there exists a subsequence, still denoted by  $(u_l)$ and a point $q\in S^4$ with $Q_{\infty}(q)>0$,
  such that the metrics $g_l$ has a measure concentration that
  $$
dv_{g_l}\to \frac{8}{3}\pi^2 \delta_q
  $$
weakly in the sense of measures, while $h_l\to c$ in $H^4(S^4,c)$
and in particular, $Q(h_l)\to 3$ in $L^2(S^4)$. Moreover, in the
latter case the conformal diffeomorphisms $\phi_l$ weakly
converges in $H^2(S^4)$ to the constant map $\phi_{\infty}=q$.
\end{Lem}

\begin{proof}
 The case 1) can be proved as Lemma 3.6 in \cite{struwe}. So we need only to
 prove the case 2). As in \cite{struwe}, we choose $q_l\in S^4$
 and radii $r_l>0$ such that
 $$
\sup_{q\in S^4}\int_{B(q,r_l)}|K(g_l)|dv_{g_l}\leq
\int_{B(q_l,r_l)}|K(g_l)|dv_{g_l}=2\pi^2,
 $$
where $B(q,r_l)$ is the geodesic ball in $(S^4,g_l)$. Then we have
$r_l\to 0$ and we may assume that $q_l\to q$ as $l\to\infty$. For
each $l$, we introduce $\phi_l$ as in Lemma 3.6 in \cite{struwe}
so that the functions
$$
\hat{u}_l=u_l\circ \phi_l+\frac{1}{4}\log (det(d\phi_l))
$$
satisfy the conformal Q-curvature equation
$$
-P_{R^4} \hat{u}_l=2\hat{Q}_l e^{4\hat{u}_l}, \; on \; R^4,
$$
where $\hat{Q}_l=Q(g_l)\circ\phi$ and $P_{R^4}$ is the Paneitz
operator of the standard Euclidean metric $g_{R^4}$. Note that for
$\hat{g}_l=\phi^*g_l=e^{2\hat{u}_l}g_{R^4}$, we have
$$
Vol(R^4, \hat{g}_l)=Vol(S^4,g_l)=\frac{8}{3}\pi^2.
$$
Arguing as in \cite{struwe}, we can conclude a convergent
subsequence $\hat{u}_l\to \hat{u}_{\infty}$ in $H^4_{loc}(R^4)$
where $\hat{u}_{\infty}$ satisfies the Liouville type equation
$$
-\Delta^2_{R^4} \hat{u}_{\infty}=\hat{Q}_{\infty}(q)
e^{4\hat{u}_{\infty}}, \; on \; R^4,
$$
with
$$
\int_{R^4}K_{\infty}(q) e^{4\hat{u}_{\infty}}dz\leq
\frac{8}{3}\pi^2.
$$

We only need to exclude the case when $Q_{\infty}(q)\leq 0$. Just
note that by (\ref{bound}) we have
$$
\log(1/\int_{R^4}f\circ \phi_l e^{4\hat{u}_l})\leq E_f(u_0).
$$
Hence, sending $l\to\infty$, we always have $f\circ\phi_l\to
f\circ\phi(q)>0$ uniformly on any compact domains of $R^4$.

The remaining part is the same as in the proof of Lemma 3.6 in
\cite{struwe}. We confer to \cite{struwe} for the full proof.

\end{proof}

We remark that some other argument can also exclude the case
$Q_{\infty}(q)< 0$. It can not occur since there is no such a
solution on the whole space $R^4$ (see also the argument in
\cite{ma}). If $Q_{\infty}(q)=0$, then
$\Delta_{R^4}\hat{u}:=\Delta_{R^4}\hat{u}_{\infty}$ is a harmonic
function in $R^4$. Let $\bar{u}(r)$ be the average of $u$ on the
circle $\partial B_r(0)\subset R^4$. Then we have
$$
\Delta_{R^4}^2 \bar{u}=0.
$$
Hence $\Delta_{R^4}\bar{u}=A+Br^{-2}$ for some constants $A$ and
$B$, where $r=|x|$. Since $\Delta_{R^4}\bar{u}$ is a continuous
function on $[0,\infty)$, we have $\Delta_{R^4}\bar{u}=A$, which
gives us that $$ \bar{u}=A+Br^2+Cr^{-2}
$$ for some constants $A,B, $ and $C$. Again, using $\bar{u}$ is regular, we have
$C=0$ and $\bar{u}=A+Br^2$ with $B<0$. However, it seems hard to
exclude this case without the use of the fact (\ref{bound}).

With this understanding, we can do the same finite-dimensional
dynamics analysis as in section 5 in \cite{struwe}. Then arguing
as in section 5 in \cite{struwe} we can prove Theorem \ref{main}.
By now the argument is well-known, so we omit the detail and refer
to \cite{struwe} for full discussion.

\end{document}